\documentclass[reqno]{amsart}

\usepackage[utf8]{inputenc}
\usepackage{amscd, amsfonts, amsmath,amssymb,amsthm, mathtools} 
\usepackage[all,color]{xy}
\usepackage[colorlinks = true,
            linkcolor = black,
            urlcolor  = blue,
            citecolor = purple,
            anchorcolor = blue]{hyperref}
\usepackage[nameinlink]{cleveref}
\usepackage{tikz-cd}
\usepackage{mathrsfs}
\usepackage{bbm}
\usepackage{stackengine,graphicx}
\usepackage{tabularx}
\usepackage{xcolor}
\usepackage{enumitem}
\usepackage{listings}
\usepackage{nicefrac}
\usepackage{quiver}
\usepackage{soul}
\usepackage{thmtools}
\usepackage{thm-restate}

\DeclareMathOperator{\Z}{\mathbb{Z}}
\DeclareMathOperator{\Q}{\mathbb{Q}}

\DeclareMathOperator{\F}{\mathbb{F}}

\DeclareMathOperator{\Div}{Div}
\DeclareMathOperator{\Br}{Br}
\DeclareMathOperator{\Gal}{Gal}

\DeclareMathOperator{\Norm}{Norm}
\DeclareMathOperator{\SL}{SL}
\DeclareMathOperator{\GL}{GL}
\DeclareMathOperator{\Cor}{Cor}
\DeclareMathOperator{\Res}{Res}

\newcommand{\abs}[1]{\left | #1 \right|}

\tikzcdset{scale cd/.style={every label/.append style={scale=#1},
    cells={nodes={scale=#1}}}}

\def \P {\mathbb{P}}

\def \Q {\mathbb{Q}}
\def \Z {\mathbb{Z}}

\def \a {\alpha}
\def \b {\beta}

\def \l {\lambda}
\def \s {\sigma}

\def \z {\zeta}

\allowdisplaybreaks 

\newtheorem{theorem}{Theorem}[section]
\newtheorem{lemma}[theorem]{Lemma}

\theoremstyle{definition}
\newtheorem{defn}[theorem]{Definition}
\newtheorem{proposition}[theorem]{Proposition}
\newtheorem{alg}[theorem]{Algorithm}
\newtheorem{cor}[theorem]{Corollary}
\newtheorem{ex}[theorem]{Example}

\theoremstyle{remark}
\newtheorem{remark}[theorem]{Remark}

\numberwithin{equation}{section}

\begin{document}
\title{Symbol Length in Brauer Groups of Elliptic Curves}

\author[M. Attanasio]{Mateo Attanasio}
\author[C. Choi]{Caroline Choi}
\author[A. Mandelshtam]{Andrei Mandelshtam}
\address[Attanasio, Choi, Mandelshtam]{Department of Mathematics, Stanford University, Stanford, CA 94305}
\email{mateoatt@stanford.edu}
\email{cchoi1@stanford.edu}
\email{andman@stanford.edu}

\author[C. Ure]{Charlotte Ure}
\address[Ure]{Department of Mathematics, University of Virginia, Charlottesville, VA 22904}
\email{cu9da@virginia.edu} 
\thanks{This paper was created as part of the 2021 Number Theory REU at the University of Virginia. We would like to thank everybody involved with organizing, lecturing, and mentoring at the REU. In particular, we would like to thank Ken Ono for managing the REU, his advice, and valuable comments on this project. We thank Andrew Sutherland and Wei-Lun Tsai for their assistance on SageMath code. We would like to thank Rachel Pries and the anonymous referee for helpful comments and suggestions. Finally, we are grateful for the generous support of the National Science Foundation (Grants DMS 2002265 and DMS 205118), the National Security Agency (Grant H98230-21-1-0059), the Thomas Jefferson Fund at the University of Virginia, and the Templeton World Charity Foundation. 
}


\date{}
\commby{}

\begin{abstract} 
    Let $\ell$ be an odd prime, and let $K$ be a field of characteristic not $2,3,$ or $\ell$ containing a primitive $\ell$-th root of unity. For an elliptic curve $E$ over $K$, we consider the standard Galois representation \[\rho_{E,\ell}: \Gal(\overline{K}/K) \rightarrow \GL_2(\F_{\ell}),\] and denote the fixed field of its kernel by $L$. Recently, the last author gave an algorithm to compute elements in the Brauer group explicitly, deducing an upper bound of $2(\ell+1)(\ell-1)$ on the symbol length in $\mathbin{_{\ell}\Br(E)} / \mathbin{_{\ell}\Br(K)}$. More precisely, the symbol length is bounded above by $2[L:K]$. We improve this bound to $[L:K]-1$ if $\ell \nmid [L:K]$. Under the additional assumption that $\Gal(L/K)$ contains an element of order $d > 1$, we further reduce it to $(1-\frac{1}{d})[L:K]$. In particular, these bounds hold for all CM elliptic curves, in which case we deduce a general upper bound of $\ell + 1$. We provide an algorithm implemented in SageMath to compute these symbols explicitly over number fields.
\end{abstract}

\maketitle

\section{Introduction} \label{sec:introduction2} 

The Brauer group of a variety is an important invariant that sheds light on the geometric and arithmetic properties of the underlying variety. For instance, Artin and Mumford used Brauer groups of varieties to provide a counterexample to the Lüroth problem in \cite{artin1972some}, and Manin showed that for a global field $K$, elements of the Brauer group can obstruct the existence of $K$-points in \cite{manin1971groupe}.

Let $F$ be a field of characteristic coprime to $\ell$ containing a primitive $\ell$-th root of unity $\zeta_{\ell}$. A \textit{symbol algebra} is a central simple algebra of the form\[(a,b)_{\ell,F} =  F \left\langle x, y: x^{\ell} = a, y^{\ell} = b, xy = \zeta_{_\ell} yx \right\rangle.\] A theorem by Merkurjev and Suslin \cite{Mer-Sus} implies that every element in the $\ell$-torsion of the Brauer group of $F$, written as $\mathbin{_{\ell}\Br(F)}$, can be described as a tensor product of symbol algebras over $F$. In particular, there is an isomorphism
$K_2(F)/\ell K_2(F) \xrightarrow{\sim} \mathbin{_\ell\text{Br}(F)},$ explicitly given by $(a,b) \mapsto (a,b)_{\ell,F}$.

For the $\ell$-torsion of the Brauer group of a field, the \textit{symbol length} is the minimal number $n$ such that some representative of every Brauer class can be written as a tensor product of at most $n$ symbol algebras. Bounds on the symbol length may be used to give upper bounds for the degree of abelian (or meta-abelian) Galois splitting fields. For a survey on symbol length, see \cite[Section 3]{auel2011open}.

The symbol length of Brauer classes of prime index is unknown in general, but is conjectured to be $1$. If $F$ is a local or global field containing a primitive $\ell$-th root of unity, then Albert, Brauer, Hasse, and Noether showed that every algebra of exponent $\ell$ is cyclic and has symbol length 1 \cite{albert1932determination, brauer1932beweis}. For $F$ a $C_2$ field containing a certain primitive root of unity, Artin showed that every Brauer class of degree $2$ or $3$ has symbol length 1 \cite{artin917brauer}. Similar results have been proven for Brauer classes defined over a function field of an $\ell$-adic curve $F$. Saltman showed that every Brauer class defined over $F$ of prime index coprime to $\ell$ is cyclic \cite{saltman2007}. Over the same such $F$, work by Suresh for prime order \cite{suresh2010}, and Brussel, McKinnie, and Tengan \cite{brussel2016cyclic} for order coprime to $\ell$, showed that every Brauer class has symbol length at most 2. Matzri gave bounds on the symbol lengths of elements in the Brauer group of a field that contains a $C_m$ field \cite{Matzri}.

The Brauer group $\Br(E)$ of an elliptic curve $E$ defined over a field $K$ is a torsion abelian group and a subgroup of the Brauer group of its function field $\Br(K(E))$. It is also isomorphic to the unramified Brauer group of $K(E)$ by purity. Applying the Merkurjev-Suslin theorem in the setting $F = K(E)$, one can express $\mathbin{_{\ell}\Br(E)}$ as a tensor product of symbol algebras. A description of $\mathbin{_{\ell}\Br(E)}$ in terms of generators and relations is given by Skorobogatov in \cite[Chapter 4]{sko-torsors} provided that the $\ell$-torsion of $E$ is $K$-rational. When $E(K)$ does not contain $E[\ell]$ and $K$ is not of characteristic $2$ or $3$, the explicit generators and relations are given for $\ell=2$ by Chernousov and Guletskii in \cite{chern-gul-2tor} and for odd primes $\ell$ coprime to the characteristic of $K$ by the last author in \cite{ure2019prime}. We will rely on these algorithms to obtain bounds on the symbol length.

Let $\ell$ be an odd prime, and let $K$ be a field of characteristic not 2 or 3 and coprime to $\ell$ containing an $\ell$-th primitive root of unity $\zeta_{\ell}$. We consider an elliptic curve $E$ defined over $K$. Under these assumptions, we have \[\mathbin{_{\ell}\Br(E)} = \mathbin{_{\ell}\Br(K)} \oplus I,\] where $I$ is the tensor product of certain explicit symbol algebras \cite{sko-torsors, chern-gul-2tor, ure2019prime}. More details are given in \Cref{subsec:epsilon}. Denote by \[\rho_{E,\ell}: \Gal(\overline{K}/K) \to \text{Aut}(E[\ell]) \left( \xrightarrow{\sim} \GL_2(\F_{\ell})\right)\] the standard Galois representation obtained by letting elements of the Galois group act on the $\ell$-torsion points of $E$. Recall that since $E[\ell] \simeq \Z/\ell\Z \times \Z/\ell\Z$, picking a $\Z/\ell\Z$-basis $P,Q$ of $E[\ell]$ gives us a natural map from $\Gal(\overline{K}/K) \to \GL_2(\F_{\ell})$. Denote by $L$ the fixed field of the kernel of $\rho_{E,\ell}$.  Equivalently, $L$ is the smallest Galois extension of $K$ such that the $\ell$-torsion points of $E$ are $L$-rational. 

 The algorithm presented by the last author in \cite{ure2019prime} implies that the symbol length in $\mathbin{_{\ell}\Br(E)}/\mathbin{_{\ell}\Br(K)}$ is bounded above by $2[L:K]$ if $\ell \nmid [L:K]$ and $\frac{2}{\ell} [L:K]$ if $\ell \mid [L:K]$. In particular, since $G$ is a subgroup of $\mathrm{SL}_2\left( \mathbb{F}_\ell\right)$, the symbol length is always bounded above by $2 (\ell-1) (\ell+1)$ \cite[Corollary 1.2]{ure2019prime}. If $\ell \nmid [L:K]$, we improve this bound significantly in a series of statements. Write $G := \Gal(L/K)$ and assume throughout that $\ell \nmid [L:K]$.

\begin{restatable}{theorem}{halfbound} \label{thm:half_bound}
    Suppose $\ell \nmid [L:K]$ and $[L:K] > 2$. Then the symbol length of $\mathbin{_{\ell}\Br(E)}/\mathbin{_{\ell}\Br(K)}$ is bounded above by $[L:K]-1$.
\end{restatable}

When there exists some intermediate field $K \subset K' \subset L$, we can further improve our bound on the symbol length of $\mathbin{_{\ell}\Br(E)} / \mathbin{_{\ell}\Br(K)}$.

\begin{restatable}{theorem}{subfieldbound}\label{thm:subfield-bound}
    Suppose $\ell \nmid [L:K]$, $[L:K]>2$, and $G$ contains a subgroup of order $d>1$. Then the symbol length of $\mathbin{_{\ell}\Br(E)}/\mathbin{_{\ell}\Br(K)}$ is bounded above by $\left(1-\frac{1}{d}\right)[L:K]$.
\end{restatable}

\begin{remark}
    When $K$ is a local or global field, we obtain the following bounds on the symbol length of $\mathbin{_{\ell}\Br(E)}$. By \cite{albert1932determination}, the symbol length of $\Br(K)$ is $1$, so \Cref{thm:half_bound} implies that when $\ell \nmid [L:K]$ and $[L:K]>2$, the symbol length of $\mathbin{_{\ell}\Br(E)}$ is bounded above by $[L:K]$. If we additionally assume that $G$ contains an element of order $d>1$, then \Cref{thm:subfield-bound} implies that the symbol length is bounded above by $\left(1-\frac{1}{d}\right)[L:K] + 1$.
\end{remark}

Theorem \ref{thm:subfield-bound} is best applicable when the Galois group $G$ has an element of order $2.$ In this case, the symbol length becomes bounded by $\frac{1}{2}[L:K].$

\begin{restatable}{cor}{NCart}\label{cor:NCart-bound}
If $K$ is a number field and the image of $\rho_{E, \ell}$ lies inside the normalizer of a Cartan subgroup, then $\mathbin{_\ell \text{Br}(E)}/\mathbin{_\ell \text{Br}(K)}$ has symbol length bounded above by $\ell - 1$ or $\ell + 1$, corresponding to a split and nonsplit Cartan respectively.
\end{restatable}
We note that the assumptions of \Cref{cor:NCart-bound} are satisfied by CM elliptic curves over a number field $K$, and we refer the reader to \Cref{rem:CM-image} for a complete discussion.

This paper is organized as follows. In \Cref{sec:preliminaries}, we introduce preliminary definitions and describe the algorithm from \cite{ure2019prime} to explicitly compute the generators of $\mathbin{_{\ell}\Br(E)}/\mathbin{_{\ell}\Br(K)}$. In \Cref{sec:image-norm}, we prove our first bounds on the symbol length. In \Cref{sec:bounds}, we improve these bounds on the symbol length and prove our main theorems. We conclude with some illustrative examples of our main theorems in \Cref{sec:examples} and compute bounds on the symbol length of $\mathbin{_{5}\Br(E)}/\mathbin{_{5}\Br(K)}$ for a CM elliptic curve $E$.

Throughout this paper, let $\ell$ be an odd prime, $K$ a number field of characteristic different from $2$, $3$, and $\ell$ containing a primitive $\ell$-th root of unity $\zeta_\ell$. Denote a separable closure of $K$ by $\overline{K}$ and for any separable extension $K'$ of $K$ let $G_{K'}$ be the absolute Galois group $\Gal(\overline{K}/K')$. Fix an elliptic curve $E$ over $K$ and let $L$ be the smallest field extension of $K$ so that the $\ell$-torsion $E[\ell]$ is $L$-rational. 

\section{Preliminaries} \label{sec:preliminaries} 

We begin by discussing the theory needed to bound the symbol length in the Brauer group of an elliptic curve. 

\subsection{Elliptic Curves and Galois Representations} 
Fix a basis $P, Q$ of the torsion subgroup $E[\ell]$ so that the Weil-pairing $\text{e}(-,-)$ satisfies $\text{e}(P,Q) = \zeta_\ell$. Recall that by letting elements of the Galois group act on $E[\ell]$, we obtain the standard Galois representation by $\ell$: \[ \rho_{E,\ell}: \Gal(\overline{K}/K) \to \GL_2(\F_{\ell}).\]

\begin{remark}\label{imagegroup}
Since we assume that $K$ contains a primitive $\ell$-th root of unity $\zeta_{_\ell}$, we have $\mathrm{Im } \rho_{E, \ell} \le SL_2(\mathbb{F}_{\ell})$. To see this, let $\sigma \in G_K$ with $\rho_{E, \ell}(\sigma) = \begin{pmatrix} a & b \\ c & d \end{pmatrix}$. Noting that the Weil pairing commutes with the Galois action through its determinant \cite[Chapter 3, Section 8]{silverman2009arithmetic}, we have {\small $$\zeta_{_\ell} = \sigma(\zeta_{_\ell}) = \sigma \text{e}(P, Q) = \text{e}(\sigma(P), \sigma(Q)) = \text{e}(aP + cQ, bP + dQ) = \text{e}(P, Q)^{ad-bc} = \zeta_{_\ell}^{ad-bc},$$} so $ad - bc = 1\mod \ell$ and $\rho_{E, \ell}(\sigma) \in SL_2(\F_{\ell})$.
\end{remark}

By examining admissible subgroups of $\GL_2(\F_{\ell})$, we can determine the structure of $G$, which will play a large role in improving bounds on the symbol length. A \textit{Borel subgroup} of $\GL_2(\mathbb{F}_{\ell})$ is any group conjugate to the group of matrices of the form $
        \begin{pmatrix} a & b \\ 0 & d\end{pmatrix}.$
A \textit{Cartan subgroup} of $\GL_2(\mathbb{F}_{\ell})$ is a maximal abelian subgroup that is semisimple. A Cartan subgroup is of index 2 in its normalizer. There are two types of Cartan subgroups. A \textit{split Cartan subgroup} is a Cartan subgroup that is conjugate to the subgroup of diagonal matrices of the form $ \begin{pmatrix} a & 0 \\ 0 & d \end{pmatrix},$ and is isomorphic to $\F_{\ell}^{\times} \times \F_{\ell}^{\times}$. Let $\varepsilon \in \F_{\ell}^{\times}$ be a non-square. A \textit{nonsplit Cartan subgroup} is a Cartan subgroup that is conjugate to the group of matrices of the form $ \begin{pmatrix} a & \varepsilon b \\ b & a \end{pmatrix},$ and is isomorphic to $\F_{\ell^2}^{\times}$  \cite{lang2012introduction, lmfdb}. If $G \neq \SL_2(\mathbb{F}_{\ell})$, then it satisfies one of the following cases \cite[p. 261]{serre1979}:
    \begin{enumerate}
        \item $G \subset $ Borel,
        \item $G \subset $ Cartan,
        \item $G \subset $ normalizer of a Cartan but not the Cartan,
        \item $PG \cong S_4$, where $PG$ denotes the image of $G$ under the projectivization $\GL_2(\F_{\ell}) \mapsto \mathrm{PGL}_2(\F_{\ell})$.
    \end{enumerate}

\begin{remark}\label{rem:CM-image}
By \cite{lang2012introduction} we know that each of the three cases above corresponds to $G$ having a different image $PG$ in the projectivization $\mathrm{PGL}_2(\F_\ell)$. If we assume that $E$ has CM, by comparing the number of order two elements in each of these images with the expected number of order two elements in $PG$ from the Chebotarev density theorem, we see that $G$ is always contained in the normalizer of a Cartan but not the Cartan. Hence if $E$ is a CM elliptic curve, we always have $\ell \nmid [L:K]$.

\end{remark}

\subsubsection{Divisors}
Recall that for a point $P \in \overline{E}$, we denote by $t_P$ any function in $\overline{K}(E)$ with divisor equal to $\ell[P] - \ell[\mathcal{O}]$. Recall also that $\sigma(t_P) = t_{\sigma(P)}$ for any $\sigma \in G_K$. Denote by $L_{P,Q}$ the line in $\P^2$ which passes through $P$ and $Q$. The divisor of $L_{P,Q}$ is equal to $[P] + [Q] + [-P-Q] - 3[\mathcal{O}]$.
\begin{proposition} \label{prop:tPmult}
Up to $\ell$th powers in $L(E)$, we may assume that  \[ t_P^{-1} = t_{-P} \qquad \text{and}\qquad t_P \cdot t_Q = t_{P+Q}.\]
\end{proposition}
\begin{proof}
These identities follow from 
\[\Div\left(\tfrac{L_{P,-P}^\ell}{t_P}\right)=\ell[-P] - \ell[\mathcal{O}] \qquad \text{and} \qquad \Div\left(\tfrac{t_P\cdot t_Q \cdot L_{-P,-Q}^\ell}{L_{P,-P}^\ell \cdot L_{Q,-Q}^\ell}\right) = \ell[P + Q] - \ell[\mathcal{O}].\qedhere \]
\end{proof}

\subsection{Brauer Groups and Group Cohomology}
We review some facts about cohomology. For more details, see \cite{neukirch2013cohomology} and \cite{Lichtenbaum-duality}.
\subsubsection{Brauer Groups}
Recall that for a field $F$ with separable closure $\overline{F}$, the Brauer group is $H^2(G_F, \overline{F})$. The Brauer group $\Br(E)$ of an elliptic curve $E$ can also be defined in terms of \'etale cohomology. However, we forgo this definition in favor of identifying $\Br(E)$ as a subgroup of the Brauer group of its function field $K(E)$. We give this identification explicitly as follows. 

Denoting by $\text{Div}(\overline{E})$ the divisors of $\overline{E}$, we have a map $\overline{K}(E)^\times \to \text{Div}(\overline{E})$ by sending each function to its divisor. This induces a map in cohomology as \[H^2(G_K, \overline{K}(E)^\times) \to H^2(G_K, \text{Div}(\overline{E})).\]
The kernel of this map is $\Br(E)$. One can verify that this is equivalent to the \'etale cohomology definition in \cite[Section 2]{Lichtenbaum-duality}.

\subsubsection{Restriction and Corestriction}
If $\mathcal{G}$ is any group, $\mathcal{H}$ is a normal finite index subgroup of $\mathcal{G}$, and $A$ is a $\mathcal{G}$-module, then we have the corestriction map $\Cor: H^n(\mathcal{H},A) \to H^n(\mathcal{G},A)$ and restriction map $\Res: H^n(\mathcal{G},A) \to H^n(\mathcal{H},A)$.

\begin{proposition} \label{prop:cor-is-norm} \cite{neukirch2013cohomology}
The map \[\Cor \circ \Res: H^n(\mathcal{G},A) \to H^n(\mathcal{G},A)\] is equal to the map of multiplication by $[\mathcal{G}:\mathcal{H}]$. When $n=0$ so that $H^0(\mathcal{G},A) = A^\mathcal{G}$, the elements of $A$ fixed by $\mathcal{G}$, the map \[\Res \circ \Cor: H^0(\mathcal{H},A) \to H^0(\mathcal{H},A)\] sends an element $a$ to $\sum_{s \in \mathcal{G}/\mathcal{H}}s \cdot a$.
\end{proposition}
Thus we see that corestriction acts like a norm map with respect to the group $\mathcal{G}/\mathcal{H}$. When $\mathcal{G} = G_K$ and $\mathcal{H} = G_L$ are absolute Galois groups of a field extension $L/K$, we write these maps as $\Cor_{L/K}$ and $\Res_{K/L}$. However, when the fields $L$ and $K$ are implied, we will ignore the subscripts. Using explicit descriptions of these maps to bound the symbol length will be the chief goal of this paper. 

\subsubsection{Galois Action on Symbols}
For any element $\sigma \in G_K$ we have an action $\sigma^*$ on the Brauer groups $\Br(L(E))$ that fixes $\Br(K(E))$. Furthermore, the elements of $G_L$ act trivially on $\text{Br}(L(E))$ and $\text{Br}(K(E))$, so in fact we have an action of $G_K/G_L = \Gal(L/K) = G$, commuting with corestriction \cite[Chapter 1]{neukirch2013cohomology}. This action restricts to give us an action of $G$ on $\text{Br}(E_L)$ and $\text{Br}(E)$. We can explicitly determine the action on symbols by the following lemma.
\begin{lemma}\cite[p. 85]{draxl-skew} \label{lem:action-on-symbs}
For any $\sigma \in \text{Gal}(\overline{K(E)}/K(E))$, the action of $\sigma^*$ on cohomology groups is given explicitly by \[\sigma^* (\a,\b)_{\ell, L(E)} = (\sigma(\a), \sigma(\b))_{\ell, L(E)}.\]
\end{lemma}
In our setting, since we only consider the action of elements of $G$, which act trivially on $\Br(K(E))$, we observe the following.

\begin{remark} \label{rem:cor}
For any $(\a, \b) \in \Br(E_L)$, we have {\small \[\Cor(\sigma(\a), \sigma(\b))_{\ell, L(E)} = \Cor\left(\sigma^*(\a,\b)_{\ell, L(E)}\right) = \sigma^*\Cor(\a,\b)_{\ell, L(E)} = \Cor(\a,\b)_{\ell, L(E)}.\]}
\end{remark}

\begin{remark} \cite[p. 209]{serre1979}
    If $\a \in K(E)^{\times}$, $\b \in L(E)^{\times}$, we have \[\Cor_{L/K}(\a, \b)_{\ell, L(E)} = \left(\a, \Norm_{L(E)/K(E)}(\b)\right)_{\ell, K(E)}.\]
\end{remark}

\subsection{Explicit Descriptions of Restriction and Corestriction} 

We now explicitly apply the Galois action on cohomology. By our assumption $\ell \nmid [L:K]$ and \Cref{prop:cor-is-norm}, we see that for any $G_K$-module $A$, the corestriction map $\mathbin{_\ell H^n(G_L,A)} \to \mathbin{_\ell H^n(G_K, A)}$ is surjective.

\subsubsection{Describing Restriction}
We will explicitly describe the image of the restriction $H^1(G_K, E[\ell]) \to H^1(G_L,E[\ell])$.
Since corestriction is surjective, it is equal to the image of $\Res \circ \Cor$.
\begin{defn}
    Denote by $\mathcal{N}$ the map \[ \Res \circ \Cor: H^1(G_L,E[\ell]) \to H^1(G_L,E[\ell]).\] For a different choice of subfield $K' \subset L$ with $H = \Gal(L/K')$, we write \[\Res_{L/K'} \circ \Cor_{L/K'} = \mathcal{N}_H.\]
\end{defn}
If we pick an $\mathbb{F}_\ell$-basis $P$, $Q$ of $E[\ell]$, then we have an isomorphism \begin{equation}\label{eq:Kummer} H^1(G_L,E[\ell]) \cong_{P,Q} H^1(G_L, \Z/\ell\Z \times \Z/\ell \Z) \cong H^1(G_L, \Z/\ell\Z) \times H^1(G_L, \Z/\ell\Z ). \end{equation} 
Kummer theory gives us an isomorphism of $H^1(G_L, \Z/\ell\Z)$ with $L^\times/(L^\times)^\ell$. Since we have an action of $G$ on $H^1(G_L, E[\ell])$, we may therefore transport it via these isomorphisms to an action of $G$ on $L^{\times}/(L^{\times})^\ell \times L^{\times}/(L^{\times})^\ell$. Explicitly, this action is given by \[ g \cdot (a,b) = \left(\left(g^{-1}(a)\right)^{c_1^g}\left(g^{-1}(b)\right)^{c_3^g},\left(g^{-1}(a)\right)^{c_2^g}\left(g^{-1}(b)\right)^{c_4^g}\right),\] where $g^{-1} = \begin{pmatrix} c_1^g & c_3^g\\ c_2^g & c_4^g \end{pmatrix}$ in terms of the basis $P$, $Q$ of $E[\ell]$ \cite[Remark 5.4]{ure2019prime}.

\begin{defn}\label{def:funny-norm}
    We denote by $\mathcal{N}_{P,Q}$ or $\mathcal{N}_{G,P,Q}$ the composition \[\tfrac{L^{\times}}{(L^{\times})^\ell} \times \tfrac{L^{\times}}{(L^{\times})^\ell} \xrightarrow{\sim_{P,Q}}H^1(G_L, E[\ell]) \xrightarrow{\mathcal{N}_G} H^1(G_L, E[\ell])\xrightarrow{\sim_{P,Q}} \tfrac{L^{\times}}{(L^{\times})^\ell} \times \tfrac{L^{\times}}{(L^{\times})^\ell},\]
    where $\sim_{P,Q}$ comes from the map above in \cref{eq:Kummer} and Kummer theory. We denote the image of the map $\mathcal{N}_{P,Q}$ by {\small \[M_{G,P,Q} =  \left\{ \mathcal{N}_{P,Q}(a,b)= \prod_{g \in G}g \cdot(a,b): a, b \in L^{\times}/(L^{\times})^\ell \right\}.\]}
\end{defn}
Note that this map is a norm with respect to the action of $G$ on $\tfrac{L^{\times}}{(L^{\times})^\ell} \times \tfrac{L^{\times}}{(L^{\times})^\ell}$.

\subsubsection{Describing Corestriction}
Our next goal is to describe the image of the corestriction map in the case of symbol algebras. We can explicitly describe the corestriction of a pair of symbols $(\a,\b) \in K_2(F)$ using the proposition from \cite[Section 3]{Rosset-Tate-1983}. By Corollary 2 of the same section, this restricts to give us the corestriction on the $\ell$-torsion of Brauer groups.
We reproduce the algorithm here for clarity of exposition.

\begin{defn}
    For a polynomial $p(t) = a_nt^n + a_{n-1}t^{n-1}+ \cdots + a_mt^m$ of degree $n$, define $p^{*}(t) = \frac{p(t)}{a_m t^m}$ and $c(p) = (-1)^na_n$.
\end{defn}

\begin{alg}[Rosset-Tate Algorithm {\cite[p. 44]{Rosset-Tate-1983}} \label{rosset-tate}]
Let $F'/F$ be a finite extension of fields over $K$, and let $x, y \in (F')^*$. Let $g(t) \in F[t]$ be the minimal polynomial of $x$ over $F$, and let $f(t) \in F[t]$ be the polynomial of smallest degree such that $\text{Norm}_{F'/F(x)}(y) = f(x)$. Define a sequence of nonzero polynomials $g_0, \dots, g_m$ as follows. Set $g_0 = g$, $g_1 = f$, and for $i \ge 1$, let $g_{i+1}$ be the remainder of the division of $g_{i-1}^{*}$ by $g_i$ as long as $g_i \neq 0$. Then the corestriction of the symbol $(x,y)_{\ell,F'}$ is given explicitly by $$ \Cor_{F'/F}(x,y)_{F'}=-\sum_{i=1}^{m}(c(g_{i-1}^*), c(g_{i}))_{F}.$$
Note that the degree of $g_i$ strictly decreases on each step, so $m \leq \deg g_0 \leq [F':F]$.
\end{alg}

\subsection{Brauer Group of an Elliptic Curve} \label{subsec:epsilon}

We seek to understand the prime torsion of $\Br(E)$ through the exact sequence 

\[\begin{tikzcd}[ampersand replacement=\&]
	0 \& {\mathbin{_\ell}\text{Br}(K)} \& {\mathbin{_\ell}\text{Br}(E)} \& {\mathbin{_\ell H^1(G_K, \overline{E})}} \& 0
	\arrow[from=1-4, to=1-5]
	\arrow["r"', from=1-3, to=1-4, swap]
	\arrow["i"', from=1-2, to=1-3, swap]
	\arrow[from=1-1, to=1-2]
\end{tikzcd}.\]
This exact sequence comes from the Hochschild-Serre spectral sequence, with the map $i$ taking the class of an algebra $A$ to the class of $A \otimes K(E)$. For details, see \cite[p. 63]{sko-torsors} or \cite[Section 2]{Lichtenbaum-duality}.

We will now describe a split to $r$ that gives us the desired decomposition of ${\mathbin{_\ell}\text{Br}(E)}$. For this, we will use a map $\varepsilon: H^1(G_K, E[\ell]) \to \mathbin{_\ell}\Br(E)$. For reasons given by \cite{chern-gul-2tor} and \cite{ure2019prime}, this will induce a split to $r$, allowing us to write $\mathbin{_\ell \text{Br}(E)} = \mathbin{_\ell \text{Br}(K)} \oplus \text{Im}(\varepsilon)$.  There is an interpretation of this map as a cup product given in \cite{sko-cup-gens}. By \cite{lenstra-k2}, when $K$ is a number field all elements of $\text{Br}(K)$ are symbols, and in general, $\Br(K)$ will not be too complicated, so we instead examine the image of $\varepsilon$.

In the special case where $K$ contains all $\ell$-torsion points (i.e., $K=L$), we can define our map $\varepsilon$ as follows.
First, fix a basis $P, Q$ for $E[\ell]$, which gives us an isomorphism $E[\ell] \xrightarrow{\sim} \Z/\ell\Z \times \Z/\ell\Z$ as $G_K$-modules. Then there is an isomorphism from $H^1(K,E[\ell])$ to $\left((K^\times/(K^\times)^\ell\right)^2$ by Kummer theory. The map $\varepsilon$ is given by \[(a,b) \mapsto (a, t_P)_{\ell, K(E)} \otimes (b,t_Q)_{\ell, K(E)}.\]

In the more general case where $K \neq L$, but we still have that $\ell \nmid [L:K]$, so that restriction is injective and corestriction is surjective, we define $\bar{\varepsilon}$ as before and use the following diagram to define $\varepsilon$ as the composition $[L:K] \circ \varepsilon = \Cor \circ \bar{\varepsilon} \circ \Res$ \cite{ure2019prime}:

\[\scalebox{0.75}{\begin{tikzcd}[ampersand replacement = \&, column sep=small, row sep=small]
	\&\&\& 0 \\
	\&\&\& {E_K/\ell E_K} \\
	\&\&\& {H^1\left(G_K, E[\ell]\right)} \&\& 0 \\
	0 \& {\mathbin{_\ell\text{Br}}(K)} \& {\mathbin{_\ell\text{Br}}(E)} \& {\mathbin{_\ell H^1\left(G_K, \overline{E}\right)}} \& 0 \& {E_L/\ell E_L} \& {} \\
	\&\&\& 0 \&\& {H^1\left(G_L, E[\ell]\right) \cong_{P,Q} \left((L^\times/(L^\times)^\ell\right)^2} \& {} \\
	\&\& 0 \& {\mathbin{_\ell\text{Br}}( L)} \& {\mathbin{_\ell\text{Br}}(E_L)} \& {\mathbin{_\ell H^1\left(G_L, \overline{E}\right)}} \& 0 \\
	\&\&\&\&\& 0
	\arrow[from=1-4, to=2-4]
	\arrow[from=2-4, to=3-4]
	\arrow[from=3-4, to=4-4]
	\arrow[from=4-4, to=5-4]
	\arrow[from=4-3, to=4-4]
	\arrow[from=4-2, to=4-3]
	\arrow[from=4-1, to=4-2]
	\arrow[from=6-6, to=6-7]
	\arrow[from=6-5, to=6-6]
	\arrow[from=6-4, to=6-5]
	\arrow[from=4-4, to=4-5]
	\arrow["\Res", color={rgb,255:red,214;green,92;blue,92}, curve={height=-12pt}, from=3-4, to=5-6]
	\arrow["\Cor", color={rgb,255:red,214;green,92;blue,92}, curve={height=-12pt}, from=6-5, to=4-3]
	\arrow[from=6-3, to=6-4]
	\arrow[from=3-6, to=4-6]
	\arrow[from=4-6, to=5-6]
	\arrow[from=5-6, to=6-6]
	\arrow[from=6-6, to=7-6]
	\arrow["{\bar{\varepsilon}}"', from=5-6, to=6-5]
	\arrow["\varepsilon"', dashed, from=3-4, to=4-3]
\end{tikzcd}}\]

Remark that the first vertical exact sequence of the diagram comes from the long exact sequence in cohomology arising from the short exact sequence of $G_K$-modules: \[0 \to E[\ell] \to \overline{E} \xrightarrow{ \ell} \overline{E} \to 0\] and similarly for the second vertical exact sequence. Following the diagram above, we take $(\alpha,\beta) \in \left( L^\times /(L^\times)^\ell\right)^2$ in the image of the restriction map and apply $\Cor \circ \overline\varepsilon$ to it. Specifically, for any element $(\a,\b) \in M_{G,P,Q}$, we use \Cref{rosset-tate} to compute the corestriction of $\overline{\varepsilon}(\a, \b) = (\a, t_P)_{L(E)} \otimes (\b, t_Q)_{L(E)}$. We use this to give bounds on the symbol length in \Cref{thm:half_bound} and \Cref{thm:subfield-bound}.

\begin{remark}
    When we have $\ell\mid[L:K],$ one can use the restriction-inflation sequence to define $\bar{\varepsilon}$ over $K'$ with $[L:K']=\ell$, and then use the previous construction to pass from the map $\bar{\varepsilon}$ over $K'$ to the $\varepsilon$ over $K$ \cite[Theorems 5.2-5.15]{ure2019prime}. However, our setup no longer applies since the extension $K'$ over $K$ is not necessarily Galois. Hence we restrict ourselves to the case $\ell \nmid [L:K]$.
\end{remark}

\section{Image of the Restriction Map to $H^1( L, E[\ell])$}
\label{sec:image-norm} 

Our first objective is to understand the image of the restriction $H^1(G_K, E[\ell]) \to H^1(G_L,E[\ell])$. As mentioned earlier, we will use the fact that if we pick a basis $P,Q$ of $E[\ell]$, we may apply isomorphisms from Kummer theory to understand the image as pairs of elements $(\a,\b) \in L^{\times}/(L^{\times})^\ell \times L^{\times}/(L^{\times})^\ell$. To prove our first bounds on the symbol length, we require the following lemma.

\begin{lemma}\label{lin-alg}
    Let $A \leqslant  \SL_2(\overline{\F}_\ell)$, with $\ell \neq 2$ and $1 \neq \abs{A} <\infty$. Then $$\sum_{a \in A}a = 0,$$ where the sum is an element of the matrix ring $M_2(\overline{\F}_\ell)$.
\end{lemma}

\begin{proof}
We first prove this in the case where $A$ is a cyclic group of prime order $p$.

Let $\sigma$ be a generator of $A$. Assume first that the generator $\sigma$ of $A$ is in Jordan normal form. We have two cases: \[\sigma_1 = \begin{pmatrix} \lambda_1 & 0 \\ 0 & \lambda_2 \end{pmatrix} \qquad \text{ or } \qquad \sigma_2 = \begin{pmatrix} \lambda & 1\\ 0 & \lambda \end{pmatrix},\]
where $\lambda_1^p = \lambda_2^p = \lambda^p = 1$. Furthermore, the latter case occurs only when $p=\ell$, since  \[\begin{pmatrix} 1 & 0\\ 0 & 1 \end{pmatrix}=\sigma_2^p = \begin{pmatrix} \lambda^p & p\lambda^{p-1}\\ 0 & \lambda^p \end{pmatrix}\]

In the first case, we must have $\l_1, \l_2\neq1$ since $\det \sigma = 1$ as $\sigma \in \SL_2(\overline{\F}_\ell)$. Hence,  \[\sum_{0 \leq n < p}\sigma_1^n = \begin{pmatrix} 1 + \lambda_1 + \dots + \lambda_1^{p-1} & 0 \\ 0 &  1 + \lambda_2 + \dots + \lambda_2^{p-1} \end{pmatrix} = \begin{pmatrix} \frac{\lambda_1^p - 1}{\lambda_1 - 1} & 0 \\ 0& \frac{\lambda_2^p - 1}{\lambda_2 - 1}\end{pmatrix} = 0.\]

In the second case, we observe the polynomial identity $x^\ell-1=(x-1)^\ell$ in $\overline{\F}_\ell,$ implying that the unique (multiple) root of this polynomial is $\l=1.$ Then a direct computation shows that  
\[
\sum_{0\leq n<\ell}\s_2^n=\begin{pmatrix}\ell&\nicefrac{\ell(\ell+1)}{2}\\0&\ell\end{pmatrix}=0.
\]

If $\sigma$ is not in Jordan form, we know that $g\sigma g^{-1}$ is in Jordan form for some $g \in \GL_2(\overline{\mathbb{F}}_\ell)$.
The theorem holds in this case by the following calculation: {\small \[\sum_{0 \leq n < p}\sigma^n = g^{-1}\left(\sum_{0 \leq n < p} g\sigma^n g^{-1} \right)g= g^{-1}\left(\sum_{0 \leq n < p} \left(g\sigma g^{-1}\right)^n \right)g= 0.\]}

Finally, for $A$ any nontrivial finite subgroup of $\SL_2(\overline{\F}_\ell)$, we note that $\abs{A}$ is divisible by some prime $p$, so it contains a cyclic subgroup $\langle\sigma\rangle$ of order $p$. Denoting by $S$ a set of representatives of left cosets of $\langle\sigma\rangle$, we have {\small \[\sum_{a \in A}a = \sum_{s \in S}\sum_{0 \leq n < p}s\sigma^n = \sum_{s \in S}s\left(\sum_{0 \leq n < p}\sigma^n \right)= 0. \qedhere \]}
\end{proof}
Using \Cref{lin-alg} we can determine the image of the restriction map $M_{G,P,Q}$ defined in \Cref{def:funny-norm}. Recall that the algorithm in \cite{ure2019prime} implies a bound on the symbol length of $2[L:K]$ if $\ell \nmid [L:K]$ and $\frac{2}{\ell}[L:K]$ if $\ell \mid [L:K]$. Combining the following result carefully with the algorithm, we may reduce these bounds significantly.  
 
\begin{proposition}\label{lem:im_norm_1}
    Let $L$ over $K$ be a finite Galois extension with $G= \Gal(L/K)$ and suppose $\ell \neq 2$, $\abs{G} \neq 1$, and $G \leq \SL_2(\F_\ell)$. Then for any $K' \subsetneq L$ we have that $(\alpha, \beta) \in M_{G, P, Q}$ satisfies \[\Norm_{L/K'}(\a)=\Norm_{L/K'}(\b)=1,\] as elements of $(K')^\times/((K')^\times)^\ell$.
\end{proposition}
\begin{proof}
Choose $(a,b) \in  L^\times/\left( L^\times\right)^\ell \times L^\times/\left( L^\times\right)^\ell$ such that $\mathcal{N}_{G,P,Q}(a,b) =(\alpha,\beta)$. We write $H$ for $\text{Gal}(L/K')$ and $S$ for a set of representatives of right cosets of $H$ in $G$. Observe that {\small \[\mathcal{N}_{G,P,Q}(a,b) = \prod_{g \in G}g\cdot(a,b) = \prod_{h \in H}h \cdot\left( \prod_{s \in S}s \cdot(a,b)\right) = \prod_{h \in H}h \cdot(a',b') = \mathcal{N}_{H,P,Q}(a',b'),\]}
where $(a',b') = \prod_{s \in S} s\cdot(a,b) \in L^\times /(L^\times)^\ell \times L^\times / (L^\times)^\ell$. Fix a basis $P, Q$ of $E[\ell]$, and let $h \in H$ with $h^{-1} = \begin{pmatrix}c_1^h & c_3^h \\ c_2^h & c_4^h \end{pmatrix}$. We have
\begin{align*}
    \mathcal{N}_{H, P, Q}(a, b) &= \left(\prod_{h \in H} (h^{-1}(a))^{c_1^h} (h^{-1}(b))^{c_3^h}, \prod_{h \in H}(h^{-1}(a))^{c_2^h} (h^{-1}(b))^{c_4^h} \right).
\end{align*}
Since the field norm $\Norm_{L/K'}$ commutes with products and is invariant under $H$-action, we see that
\begin{align*}
    \Norm_{L/K'}(\alpha) &= \left(\Norm_{L/K'}(a')\right)^{\sum_{h \in H} c_1^h} \cdot \left(\Norm_{L/K'}(b')\right)^{\sum_{h \in H} c_3^h} \\ \Norm_{L/K'}(\beta) &= \left(\Norm_{L/K'}(a')\right)^{\sum_{h \in H} c_2^h} \cdot \left(\Norm_{L/K'}(b')\right)^{\sum_{h \in H} c_4^h}.
\end{align*}
Since $H \leq  \SL_2(\F_\ell)$, $\ell \neq 2$, and $1 \neq \abs{H} <\infty$, by \Cref{lin-alg} we have $\sum_{h \in H} c_i^h = 0$. The statement of the proposition follows immediately.
\end{proof}

\begin{cor}
If $\abs{G} = \abs{\langle\sigma\rangle} = 2$, then $M_{G,P,Q}$ is exactly equal to the elements $(\a,\b)$ of $L^{\times}/(L^{\times})^\ell \times L^{\times}/(L^{\times})^\ell$ such that $\Norm_{L/K}(\a)$ and $\Norm_{L/K}(\b)$ are 1 modulo $\ell^{th}$ powers. This is because \[\mathcal{N}_{G,P,Q}(a,b) = \left(\frac{a}{\sigma(a)}, \frac{b}{\sigma(b)}\right), \] and by Hilbert's Theorem 90 \cite{hilbert}, all norm 1 elements are of this form.
\end{cor}

The fact that the image of the restriction map on any subfield will always give us elements of field norm 1 implies the following observation on how the symbol length changes upon applying the corestriction map.

\begin{lemma}\label{naive-bound-lemma}
    Suppose $\ell \nmid [L:K']$ and $(s,t) \in \Br(L(E))$ with $\Norm_{L(E)/K'(E)}{s} = 1$ modulo $\ell^{\text{th}}$ powers. Then the length of the symbol \[\Cor_{L(E)/K'(E)}(s,t)_{\ell, L(E)}\] is at most $[L:K']-1$. 

\end{lemma}
\begin{proof}
    We apply \Cref{rosset-tate}. Let $p_s := x^d + e_{d-1}x^{d-1} + \dots + e_1x + (-1)^de_0^\ell$ denote the minimum polynomial of $s$ over $K'(E)$.
    Let $g_0, g_1, \dots, g_m$ be the sequence of polynomials as in \Cref{rosset-tate} with $g_0 = p_s$. Then
  {\small  \begin{align*}
        \Cor_{L(E)/K'(E)}(s,t)_{\ell,L(E)} &= -\sum_{i=1}^{m} \left(c(g_{i-1}^{*}), c(g_i) \right)_{\ell,K'(E)} \\
        &= -((-1)^{2d}\cdot e_0^{-\ell}, c(g_1))_{\ell,K'(E)} \otimes -\sum_{i=2}^{m} \left(c(g_{i-1}^{*}), c(g_i) \right)_{\ell,K'(E)} \\
        &= -\sum_{i=2}^{m} \left(c(g_{i-1}^{*}), c(g_i) \right)_{\ell,K'(E)}.
    \end{align*}}
    \Cref{rosset-tate} implies that $m \leq \deg(g_0) \leq [L:K']$, so the length of the symbol \[\Cor_{L(E)/K'(E)}(\a, \b)_{L(E)}\] is at most $m-1 \le [L:K']-1$.
\end{proof}
We can apply this lemma to obtain our first upper bound on the symbol length of $\mathbin{_\ell \text{Br}(E)}/\mathbin{_\ell \text{Br}(K)}$.

\begin{cor}\label{cor:naive-bound}
    Suppose that $K \neq L$, that $\ell \nmid [L:K]$, and $(\a,\b) \in M_{G, P, Q}.$ Then the symbol length of  \[\Cor_{L(E)/K(E)}(\a,t_P)_{L(E)} \otimes \Cor_{L(E)/K(E)}(\b, t_Q)_{L(E)}\] and hence of $\mathbin{_\ell \text{Br}(E)}/\mathbin{_\ell \text{Br}(K)}$ is at most $2[L:K]-2$.
\end{cor}
\begin{proof}
    Suppose $\ell \nmid [L:K]$ and $(\a,\b) \in M_{G, P, Q}$. Since $K \neq L$, we may apply \Cref{lem:im_norm_1} to see that $\text{Norm}_{L/K}(\a) = \text{Norm}_{L/K}(\b) = 1$ modulo $\ell$-th powers. By \Cref{naive-bound-lemma}, the symbol length of \[\Cor_{L(E)/K(E)}(a,t_P)_{L(E)}\]
    is at most $[L:K] - 1$ and similarly for 
    $\Cor_{L(E)/K(E)}(b, t_Q)_{L(E)}$.
    Hence the symbol length of their tensor product is at most $2[L:K]-2$.
\end{proof}

\section{Bounds on the Symbol Length}\label{sec:bounds} 
In the previous section, we used \Cref{rosset-tate} to obtain a bound on the symbol length of corestriction. Our next step is to choose judiciously a basis $P, Q$ of $E[\ell]$ in order to significantly improve the bound on the symbol length of $\mathbin{_{\ell}\Br(E)}/\mathbin{_{\ell}\Br(K)}$.

\begin{proposition}\label{prop:asa} 
If $G \leq \SL_2(\F_\ell)$ and there exists $\sigma \in G$ with order $d > 2$, then there exists a choice of basis $P, Q$ such that $Q=\s(P)$ and any element $(\a,\b)\in M_{G,P,Q}$ satisfies $\b=\frac{1}{\s^{-1}(\a)}.$
\end{proposition}

\begin{proof}
    We first reduce to the case that $G$ is cyclic. Assume that the proposition holds for some subgroup $H \leq G$,
    then the same also holds for $G$. If we write $S$ for a set of representatives of right cosets of $H$, then for any $(a,b) \in L^{\times}/(L^{\times})^{\ell} \times L^{\times}/(L^{\times})^{\ell}$, we have 
    \[\mathcal{N}(a,b) = \prod_{g \in G}g \cdot(a,b) = \prod_{h \in H} h \cdot \prod_{s \in S}s \cdot(a,b) = \prod_{h \in H} h \cdot (a', b')  = (\alpha, \beta).\]
    However, this is just $\mathcal{N}_H(a',b')$, so there exists some $\sigma \in H$ such that $\beta = \frac{1}{\s^{-1}(\a)}$.  \\
    
    Suppose from now on that $G$ is cyclic of order $d$. Since $\abs{G}>2$ and there is only one element in $\SL_2(\F_\ell)$ of order equal to $2$, we may assume that $d > 2$. Each of the eigenvalues of $\s$ is a $d$-th root of unity $\z_d$. Since $\s$ has determinant $1,$ its other eigenvalue is $\z_d^{-1}.$ Note that we cannot have $\z_d = \pm1$ since $d \neq 2$. Since $\ell\nmid\abs{G},$ we must have $\s=\pm I.$ In particular, we have $\z_d \neq \z_d^{-1}$ so the invariant factor is $x^2 - (\z_d + \z_d^{-1})x + 1.$ We thus obtain a rational canonical form \[
    \s=\begin{pmatrix}0&-1\\1&\z_d+\z_d^{-1}\end{pmatrix}.
    \]
    A direct computation shows that for $k \geq 2,$ {\small \begin{align*}
    \s^k&= \begin{pmatrix} - \sum\limits_{l=0}^{k-2} \z_d^{k-2-2l} & -\sum\limits_{l=0}^{k-1} \z_d^{k-1-2l} \\  \sum\limits_{l=0}^{k-1} \z_d^{k-1-2l} & \sum\limits_{l=0}^k \z_d^{k-2l} \end{pmatrix}\\
    &=\begin{pmatrix}
    -(\z_d^{k-2}+\z_d^{k-4}+\cdots+\z_d^{4-k}+\z_d^{2-k}) & -(\z_d^{k-1}+\z_d^{k-3}+\cdots+\z_d^{3-k}+\z_d^{1-k}) \\ \z_d^{k-1}+\z_d^{k-3}+\cdots+\z_d^{3-k}+\z_d^{1-k} & \z_d^k+\z_d^{k-2}+\cdots+\z_d^{2-k}+\z_d^{-k}
    \end{pmatrix}.
    \end{align*}}
    By abuse of notation, we will denote by $\zeta_d + \zeta_d^{-1}$ both an element of $\F_\ell$ as well as any lift of this to $\Z$. Thus we may raise any element of $L^\times/(L^\times)^\ell$ to the power of $\zeta_d + \zeta_d^{-1}$ in a well-defined manner. 
    In particular, the image of $(a,b)$ is $(\a, \b)$ with
    {\small \begin{align*} 
        \alpha &= \left(a^{-1} \sigma(b)\sigma^2(a)\sigma^2\left(b^{\zeta_d  + \zeta_d^{-1}}\right) \cdots \sigma^{d-1}\left(a^{\zeta_d^{d-3} + \cdots + \zeta_d^{3-d}}\right)\sigma^{d-1}\left(b^{\zeta_d^{d-2} + \cdots + \zeta_d^{2-d}}\right)
    \right)^{-1}, \\
        \beta &= b \sigma(a) \sigma\left(b^{\zeta_d  +\zeta_d^{-1}}\right) \cdots \sigma^{d-1}\left(a^{\zeta_d^{d-2} + \cdots + \zeta_d^{2-d}}\right)\sigma^{d-1}\left(b^{\zeta_d^{d-1} + \cdots + \zeta_d^{1-d}}\right).
    \end{align*}}
    Noting the identities $\z_d^{d-1}+\cdots+\z_d^{1-d}=0$ (for even $d,$ this is because for every element $a$ in this sum there is $-a,$ and for odd $d$, all exponents occur exactly once) and $\z_d^{d-2}+\cdots+\z_d^{2-d}=-1,$ we see that $\b=\frac{1}{\s^{-1}(\a)}$ as required.
\end{proof}

\begin{remark}
If $G$ contains an element $\sigma$ of order 4, then picking a basis $P,Q$ as in \Cref{prop:asa}, we have $\beta = \sigma(\alpha)$. In this case, the rational canonical form of $\sigma$ is
\[\sigma=\begin{pmatrix}0&-1\\1&0\end{pmatrix}.
\]
    Applying $\mathcal{N}$, we get that an element $(a,b)$ is mapped to \[(\a, \b)=
    \left(\frac{a\s^3(b)}{\s^2(a)\s(b)},\frac{\s(a)b}{\s^3(a)\s^2(b)}\right).\]
    From here, it is clear that $\b = \s(\a)$.
\end{remark}

Using \Cref{prop:asa}, we can now prove our main result. 

\halfbound* 

\begin{proof}
The image of $\varepsilon$ consists of corestrictions of elements of the form $$(\a,t_P)_{\ell, L(E)} \otimes (\b,t_Q)_{\ell, L(E)}$$ for $(\a,\b) \in M_{G,P,Q}$. Fix $\sigma, P, Q$ as in \Cref{prop:asa}. Then by the remarks at the end of Section 2 and \Cref{prop:asa}, the image consists of corestrictions of elements of the form 
\[(\a, t_P)_{\ell, L(E)} \otimes ((\sigma^{-1}(\a))^{-1}, t_{\sigma(P)})_{\ell, L(E)}.\]
By \Cref{rem:cor}, the corestriction is invariant under $G$-action, so that the corestriction of $$((\sigma^{-1}(\a))^{-1}, t_{\sigma(P)})_{\ell, L(E)}$$ is equal to the corestriction of 
\[\sigma^* ((\sigma^{-1}(\a))^{-1}, t_{\sigma(P)})=
\left(\alpha^{-1}, t_{\sigma^2(P)}\right)_{\ell, L(E)} = \left(\alpha,t_{\sigma^2(P)}^{-1} \right)_{\ell, L(E)}.\]
Hence {\small \begin{align*}
\Cor\left( (\a, t_P)_{\ell, L(E)} \otimes \left(\tfrac{1}{\sigma^{-1}(\a)}, t_{\sigma(P)}\right)_{\ell, L(E)} \right)  
&= \Cor(\a, t_P)_{\ell, L(E)} \otimes \Cor \left(\alpha,t_{\sigma^2(P)}^{-1} \right)_{\ell, L(E)}\\
&=\Cor \left(\a, \tfrac{t_P}{\sigma^{2}(t_P)}\right)_{\ell, L(E)}. \end{align*}}
By \Cref{lem:im_norm_1}, $\Norm_{L/K}(\alpha) = 1$ modulo $\ell^{\text{th}}$ powers so by \Cref{naive-bound-lemma}, we see that the corestriction of this symbol is the sum of at most $[L:K] - 1$ symbols.
\end{proof}

When $L/K$ contains an intermediate field extensions $K'$, we can factor the corestriction map through $K'$ to further improve the bound and prove our second main theorem.

\subfieldbound*

\begin{proof}
Let $H$ be an order $d$-subgroup of $G$. We write $K' = L^H$ the fixed field of $H$, so that $[L:K']=d$. We will factor the corestriction map as \[\Cor_{\Br(E_L)/\Br(E)} = \Cor_{\Br(E_{K'})/\Br(E)} \circ \Cor_{\Br(E_L)/\Br(E_{K'})}.\]
Let $\s$ be an element of $G$ of order $>2$. 
By \Cref{prop:asa}, it suffices to bound the symbol length of the corestriction of symbols of the form
\[\left(\a, \frac{t_P}{\sigma^{2}(t_P)}\right)_{\ell, L(E)}=\left(\a, t_{P-\s^2(P)}\right)_{\ell, L(E)}\tag{$\ast$}\]
Again by \Cref{lem:im_norm_1}, $\Norm_{L/K'}(\alpha) = 1$ modulo $\ell^{\text{th}}$ powers, so using \Cref{naive-bound-lemma}, we can bound the symbol length of the corestriction of $(\ast)$ to $\Br(E_{K'})$ by $[L:K']-1 = d-1$. By applying \Cref{rosset-tate}, the corestriction of a symbol in $\Br(E_{K'})$ to $\Br(E)$ is the sum of at most $[K':K]$ symbols, so that the corestriction of $(\ast)$ to $\Br(E)$ is the sum of at most $(d-1)[K':K] = (1-\frac{1}{d})[L:K]$ symbols.
\end{proof}

By \Cref{rem:CM-image}, we can apply \Cref{thm:half_bound} and \Cref{thm:subfield-bound} for any CM elliptic curve over number fields. 

\NCart*

\begin{proof}
Since we assume that $\ell$ is an odd prime, we may ignore the case $[L:K]=2$. In this case, \Cref{cor:naive-bound} gives us an upper bound of 2 on the symbol length.
Recall that the order of the normalizer of a Cartan subgroup is either $2(\ell^2-1)$ or $2(\ell - 1)^2$, depending on whether we have a nonsplit or split Cartan. Since $K/\Q$ is of degree at least $\ell-1$, we therefore see that $G$ is a subgroup of a group of order $2(\ell +1)$ or $2(\ell - 1)$. If it is a strict subgroup of either, then the statement holds by \Cref{thm:half_bound}. Otherwise $G$ must contain an element of order 2, so the bound holds by \Cref{thm:subfield-bound}.
\end{proof}

\section{Examples}\label{sec:examples} 

We provide an example to illustrate \Cref{cor:NCart-bound}. Let $K$ be a number field, and consider the infinite family of CM elliptic curves \[\mathcal{F}: y^2 = x^3 + c, \qquad c \in K.\] Let $E$ be a CM elliptic curve in the infinite family $\mathcal{F}$. We compute the symbols in $\mathbin{_{5}\Br(E)} / \mathbin{_{5}\Br(K)}$.  To see explicit generators of $\mathbin{_{5}\Br(E)} / \mathbin{_{5}\Br(K)}$, as well as the code written to perform these computations, please visit \url{https://github.com/team-brauer/brauer-groups} \footnote{All computations were done using SageMath \cite{sagemath}.}. 
Our example requires the following two propositions. The first is a restatement of Theorem 1.4 (ii) in \cite{lozano2018galois}, and the second is an application of \Cref{thm:half_bound} and \Cref{thm:subfield-bound} to the case $\ell =5$. 

\begin{proposition}\label{prop:alvaro}
    \cite[Theorem 1.4(ii)]{lozano2018galois} Let $E/\Q$ be an elliptic curve with $j(E) = 0$. If $\ell \equiv 2, 5 \mod 9$, then the image of the associated Galois representation $\rho_{E, \ell}$ is contained in the normalizer of a nonsplit Cartan.
\end{proposition}

\begin{proposition}\label{prop:nonsplit}
    If the image of the Galois representation \[\rho_{E, 5}: \Gal(\overline{K}/K) \rightarrow \SL_2(\F_5)\] is contained in the normalizer of a nonsplit Cartan, then $[L:K] \mid 12$.
    In particular, the symbol length of $\mathbin{_{\ell}\Br(E)}/\mathbin{_{\ell}\Br(K)}$ is bounded above by 6. 
\end{proposition}
\begin{proof}
The order of a nonsplit Cartan subgroup is equal to $\ell^2 - 1 = 24$ since it is isomorphic to $\F_{\ell^2}^\times$. By \cite{lang2012introduction}, we know that a Cartan subgroup has index at most 2 in its normalizer, so the image of $\rho_{E,5}$ is contained inside a group of order at most 48. Therefore $[L:\Q] \mid 48$. However, since $K$ contains $\zeta_5$ we see that $[K:\Q]$ is a multiple of 4, hence $[L:K] \mid 12$. The bound on the symbol length follows from \Cref{cor:NCart-bound}.
\end{proof}
Since an elliptic curve $E$ in the infinite family of CM elliptic curves $\mathcal{F}$ satisfies $j(E) = 0$, \Cref{cor:NCart-bound}, \Cref{prop:alvaro}, and \Cref{prop:nonsplit} immediately imply that the symbol length of $\mathbin{_{5}\Br(E)}/\mathbin{_{5}\Br(K)}$ is bounded above by 6.
In the following example, we use our code to compute the symbols in $\mathbin{_{5}\Br(E)} / \mathbin{_{5}\Br(K)}$ for a specific CM elliptic curve $E$ in the infinite family $\mathcal{F}$ and illustrate \Cref{thm:subfield-bound}.

\begin{ex}
Let $\z_{5}$ be a primitive 5th root of unity, and consider the number field $K = \Q(\z_5)$. Consider the CM elliptic curve over $K$ given by the affine equation $$E: y^2 = x^3 + 10.$$ 
\end{ex}

All computations for this example were done in SageMath \cite{sagemath} and can be found at 
\url{https://github.com/team-brauer/brauer-groups}.
The division field $L=K(E[5])$ is an extension of degree $4$ defined by adjoining a root $\mu$ of the polynomial $z^4 + \left(6\zeta_5^3 - 6\zeta_5^2 - 6\zeta_5 + 6 \right)z^2 + 24\zeta_5^3 - 12\zeta_5 - 12$.
We have an element $\sigma$ of order $4$ which sends \[\mu \to \left(\tfrac12\zeta_5^2 + \tfrac12\zeta_5 + \tfrac12\right)\mu^3 + \left (-\zeta_5^3 - \zeta_5^2 + 4 \right)\mu.\]
Using SageMath, we choose a basis $P, Q$ of $E[5]$ as in \Cref{prop:asa} given by {\small \begin{align*} P &= \left (-\zeta_5^3\mu^2 - 6\zeta_5^3 - 6\zeta_5 + 2, \left(-\zeta_5^3 + \tfrac{1}{2}\zeta_5^2 - \tfrac{1}{2}\zeta_5 + 1\right)\mu^3 + (3\zeta_5^3 + 6\zeta_5^2 - 9\zeta_5 + 15)\mu \right),\\ Q &= \left (\zeta_5^3\mu^2 + 6\zeta_5^3 + 6\zeta_5^2 + 6\zeta_5 + 2 , \left(\tfrac{1}{2}\zeta_5^2 - \tfrac{1}{2}\zeta_5 \right)\mu^3 + \left (3\zeta_5^3 + 6\zeta_5^2 + 3\zeta_5 + 3 \right)\mu \right).\end{align*}}   We observe that $\sigma(P) = Q$. Choose $(\a, \b) \in M_{G, P, Q}$ to be the image of $(\zeta_5+\mu, \mu)$ after applying $\mathcal{N}_{P,Q}$. Specifically, 
{\small \begin{align*}
    \a &= (-16188/319831\zeta_5^3 + 15910/319831\zeta_5^2 - 4980/319831\zeta_5 + 14076/319831)\mu^3\\& + (32098/319831\zeta_5^3 + 11208/319831\zeta_5^2 + 30264/319831\zeta_5 + 16188/319831)\mu^2 \\&+ (-70930/319831\zeta_5^3 + 67262/319831\zeta_5^2 - 199162/319831\zeta_5 + 273158/319831)\mu \\&+ 138192/319831\zeta_5^3 - 128232/319831\zeta_5^2 + 344088/319831\zeta_5 + 390761/319831.
\end{align*}}
Observe that $\sigma(\a)^{-1} = \b$. Now by \Cref{thm:subfield-bound}, it suffices to compute the corestriction of \[(\a, t_{P}/t_{\sigma^2(P)})_{5, L(E)}.\]
By \Cref{prop:tPmult}, this is in fact equal to the corestriction of \[(\a, t_{2P})_{5, L(E)},\] since $\sigma^2(S) = 2S$ for any $S \in E[\ell]$. We calculate this corestriction by first passing through the intermediate field $K' = K(\zeta_3) = \Q(\zeta_{15})$. 
The corestriction of the symbol $(\a, t_{2P})_{5, L(E)}$ to $\Q(\zeta_{15})$ as calculated by \Cref{rosset-tate} gives us two symbols but the first is trivial as $\alpha$ has norm $1$ by \Cref{lem:im_norm_1}. The second symbol is given by 
{\tiny \begin{align*}
    (((&551984402727902346319004580861749039896316402068750345026\cdot \zeta_{15} ^7+ \\
    &134424070611518537543552069107213590707472267793980007392\cdot \zeta_{15} ^6- \cdots )\cdot x^2 y^{10}+ \cdots)\\
    /((-&134396143139335966898002857906327476642111742609439665612\z_{15}^7 -\\ 
    &358389715038229245061340954416873271045631313625172441632\z_{15}^6 +\cdots)\cdot x^2y^{10}+\cdots), \\
    ((&86565057808308862398284296950\z_{15}^7 -15301839790878217253755018864\z_{15}^6-\cdots)\cdot x^2y^4+\cdots)\\
    /((&36812512953998458768246189661\z_{15}^7 -23426144607089928307065757057\z_{15}^5 + \cdots)\cdot x^2y^4+\cdots)).
\end{align*}}
The symbol in its entirety is approximately 53000 characters long and can be viewed at
\url{https://github.com/team-brauer/brauer-groups}.
\bibliographystyle{amsplain}

\end{document}